\documentclass[11pt]{amsart}
\usepackage[margin=1in]{geometry}

\usepackage{amssymb}
\usepackage{amsthm}
\usepackage{amsmath}
\usepackage{mathrsfs}
\usepackage{amsbsy}
\usepackage[all]{xy}
\usepackage{bm}
\usepackage{hyperref}
\usepackage{tikz}
\usepackage{array}
\usepackage{float}
\usepackage{enumerate}
\usepackage{xcolor}
\usepackage{hhline}
\allowdisplaybreaks
\usepackage{cite}

\newcommand{\st}{\operatorname{st}}

\newenvironment{enumerate*}%
  {\begin{enumerate}[(I)]%
    \setlength{\itemsep}{10pt}%
    \setlength{\parskip}{0pt}}%
  {\end{enumerate}}

\newtheorem{theorem}{Theorem}[section]
\newtheorem{proposition}[theorem]{Proposition}
\newtheorem{corollary}[theorem]{Corollary}
\newtheorem{conjecture}[theorem]{Conjecture}
\newtheorem{question}[theorem]{Question}
\newtheorem{lemma}[theorem]{Lemma}

\theoremstyle{definition}
\newtheorem{definition}[theorem]{Definition}
\newtheorem{remark}[theorem]{Remark}
\newtheorem{example}[theorem]{Example}

\begin{document}

\title[]{Promotion Sorting} \keywords{}
\subjclass[2010]{}

\author[]{Colin Defant}
\address[]{Fine Hall, 304 Washington Rd., Princeton, NJ 08544}
\email{cdefant@princeton.edu}
\author[]{Noah Kravitz}
\address[]{Grace Hopper College, Zoom University at Yale, New Haven, CT 06510, USA}
\email{noah.kravitz@yale.edu}

\begin{abstract} 
Sch\"{u}tzenberger's promotion operator is an extensively-studied bijection that permutes the linear extensions of a finite poset. We introduce a natural extension $\partial$ of this operator that acts on all labelings of a poset. We prove several properties of $\partial$; in particular, we show that for every labeling $L$ of an $n$-element poset $P$, the labeling $\partial^{n-1}(L)$ is a linear extension of $P$. Thus, we can view the dynamical system defined by $\partial$ as a sorting procedure that sorts labelings into linear extensions. For all $0\leq k\leq n-1$, we characterize the $n$-element posets $P$ that admit labelings that require at least $n-k-1$ iterations of $\partial$ in order to become linear extensions. The case in which $k=0$ concerns labelings that require the maximum possible number of iterations in order to be sorted; we call these labelings \emph{tangled}. We explicitly enumerate tangled labelings for a large class of posets that we call inflated rooted forest posets. For an arbitrary finite poset, we show how to enumerate the sortable labelings, which are the labelings $L$ such that $\partial(L)$ is a linear extension. 
\end{abstract}
\maketitle

\section{Introduction}
\subsection{Background}

Let $P$ be an $n$-element poset with order relations denoted by $\leq_P$. A \emph{labeling} of $P$ is a bijection $L:P\to[n]$, where $[n]:=\{1,\ldots,n\}$. We say that a labeling $L$ is a \emph{linear extension} if $L(x)\leq L(y)$ whenever $x\leq_P y$. Let $\Lambda(P)$ denote the set of labelings of $P$, and let $\mathcal L(P)\subseteq\Lambda(P)$ be the set of linear extensions of $P$. 

In \cite{Schutzenberger1, Schutzenberger2, Schutzenberger3}, Sch\"utzenberger introduced a fascinating bijection from $\mathcal L(P)$ to $\mathcal L(P)$ known as \emph{promotion}. Subsequently, several researchers have investigated this map and found connections with various aspects of algebraic combinatorics and representation theory (see \cite{Edelman, Huang, Petersen, Rhoades, StanleyPromotion} and the references therein).   For most posets $P$, the dynamics of the promotion bijection on $\mathcal L(P)$ are not well understood. Indeed, there are only a few families of posets for which even the order of promotion is known. 

In this article, we introduce a natural extension of promotion, which we denote by $\partial$, to the set of all labelings of a finite poset. We will see that for every $n$-element poset $P$, the extended promotion map $\partial:\Lambda(P)\to\Lambda(P)$ has the following basic properties:
\begin{enumerate}
\item The restriction of $\partial:\Lambda(P)\to\Lambda(P)$ to $\mathcal L(P)$ is the usual promotion bijection on $\mathcal L(P)$.
\item We have $\partial^{n-1}(\Lambda(P))=\mathcal L(P)$. 
\end{enumerate}
Thus, the set of periodic points of $\partial$ is precisely the set of linear extensions of $P$. Extended promotion also serves as a vast generalization of the bubble sort map ${\bf B}$ studied in \cite{Chung} and \cite[pages 106--110]{Knuth2} (see Section~\ref{Sec:Basic}). 

One can view a linear extension of an $n$-element poset $P$ as a labeling that is ``sorted.'' Property (2) tells us that if we iteratively apply the extended promotion map to a labeling of $P$, then we will eventually obtain a linear extension. Thus, we can view extended promotion as a type of sorting map. It is then natural to ask how long it takes to sort various labelings. We will be primarily interested in those labelings that require $n-1$ iterations to be sorted. Because these labeling are the farthest from being sorted, we call them \emph{tangled} labelings. More generally, for $k\leq n-2$, we will consider \emph{$k$-untangled} labelings, which are labelings $L$ such that $\partial^{n-k-2}(L)\not\in\mathcal L(P)$. In this language, a labeling is tangled if and only if it is $0$-untangled.  

Let us remark that, for other noninvertible combinatorial dynamical systems, there has been a substantial amount of work aimed at understanding the elements that require the maximum, or close to the maximum, number of iterations to reach a periodic point. For the bubble sort map, this problem was investigated by Knuth \cite[pages 106--110]{Knuth2}. West and Claesson--Dukes--Steingr\'imsson studied permutations requiring close to the maximum number of iterations to be sorted by West's stack-sorting map \cite{Claessonn-4, West}. Ungar gave a rather nontrivial argument proving that the maximum number of iterations needed to sort a permutation in $S_n$ using the pop-stack-sorting map is $n-1$ \cite{Ungar}, and the permutations requiring $n-1$ iterations were investigated further by Asinowski, Banderier, and Hackl \cite{Asinowski}. Toom \cite{Toom} asked for the maximum number of iterations of the Bulgarian solitaire map that are needed to send a partition to a periodic point, and Hobby and Knuth posed a specific conjecture concerning this maximum in \cite{Hobby}. Bentz \cite{Bentz} and Igusa \cite{Igusa} independently proved this conjecture, and Etienne \cite{Etienne} generalized their solution. Griggs and Ho studied the analogous problem for the Carolina solitaire map \cite{Griggs}. 

It is also natural to ask about the image of $\partial$ and the set of sortable labelings, which we define to be the labelings $L$ such that $\partial(L)$ is a linear extension. In Section~\ref{sec:sortable}, we show that a labeling $L$ of an $n$-element poset $P$ is in the image of $\partial$ if and only if $L^{-1}(n)$ is a maximal element of $P$. We also show how to describe the preimages of an arbitrary labeling under $\partial$. This allows us to give a simple expression for the number of sortable labelings of an arbitrary poset.  

\subsection{Outline and Summary of Main Results}

In Section~\ref{Sec:Basic}, we define the extended promotion map $\partial:\Lambda(P)\to\Lambda(P)$ for every finite poset $P$. We show that one can also describe this map in terms of generalizations of the Bender-Knuth toggles that Haiman and Malvenuto--Reutenauer used to study classical promotion. We also show that a poset $P$ admits a $k$-untangled labeling (meaning $\partial^{n-k-2}(\Lambda(P))\neq\mathcal L(P)$) if and only if $P$ contains a lower order ideal of size $k+2$ that is not an antichain. In Section~\ref{Sec:Enumeration}, we explicitly enumerate the tangled labelings of a large class of posets that we call \emph{inflated rooted forest posets}. This collection of posets contains every poset whose connected components each have unique minimal elements.  In Section ~\ref{sec:sortable}, we show how to enumerate the preimages of a labeling under $L$ and give an expression for the number of sortable labelings of an arbitrary poset. In Section~\ref{sec:conclusion}, we raise some open questions for future inquiry.

\subsection{Notation and Terminology}\label{Subsec:Notation}
A \emph{lower order ideal} of a poset $P$ is a subset $Q\subseteq P$ such that for every $x\in Q$ and every $x'\in P$ with $x'\leq_P x$, we have $x'\in Q$. Similarly, an \emph{upper order ideal} of $P$ is a subset $Q\subseteq P$ such that for every $x\in Q$ and every $x'\in P$ with $x\leq_P x'$, we have $x'\in Q$.  Note that $Q$ is a lower order ideal if and only if $P \setminus Q$ is an upper order ideal.  Given two elements $x,y\in P$, we say $y$ \emph{covers} $x$ if $x<_P y$ and there does not exist $z\in P$ satisfying $x<_P z<_P y$. A poset is \emph{connected} if its Hasse diagram is connected as a graph. The \emph{connected components} of $P$ are the maximal connected subposets of $P$ (i.e., the subposets formed by the connected components of the Hasse diagram of $P$).    

Suppose $P$ is an $n$-element poset and $f:P\to\mathbb Z$ is an injective function. The \emph{standardization} of $f$, denoted $\st(f)$, is the unique labeling $L: P \to [n]$ such that for all $x,y\in P$, we have $L(x)<L(y)$ if and only if $f(x)<f(y)$. Equivalently, if $g:f(P)\to [n]$ is an order-preserving bijection, then $\st(f)=g\circ f$. Once we define the extended promotion map $\partial$, it will be convenient to write $L_{\gamma}=\partial^{\gamma}(L)$ for every labeling $L$ and every integer $\gamma\geq 0$ (where $L_0=L$).

\section{Properties of Extended Promotion}\label{Sec:Basic}

Let $L:P\to[n]$ be a labeling of an $n$-element poset $P$. Suppose $x\in P$ is not a maximal element. This means that there are elements of $P$ that are greater than $x$; among all such elements, let $y$ be the one with the smallest label (i.e., the one that minimizes $L(y)$). We call $y$ the \emph{$L$-successor} of $x$. Notice that $y$ does not necessarily cover $x$ in $P$. However, if $L$ is a linear extension, then $y$ does necessarily cover $x$. Now let $v_1=L^{-1}(1)$. If $v_1$ is not a maximal element of $P$, then let $v_2$ be its $L$-successor. If $v_2$ is not maximal, let $v_3$ be its $L$-successor. Continue in this fashion until obtaining an element $v_m$ that is maximal. We define the labeling $\partial(L)\in\Lambda(P)$ by 
\[\partial(L)(x)=\begin{cases} L(x)-1, & \mbox{if } x\not\in\{v_1,\ldots,v_m\}; \\ L(v_{i+1})-1, & \mbox{if } x=v_i\text{ for some }i\in\{1,\ldots,m-1\}; \\ n, & \mbox{if } x=v_m. \end{cases}\]

It is immediate from our definition that $\partial(L)$ is indeed a labeling of $P$. In the case in which $L$ is a linear extension, our definition agrees with the usual definition of promotion. The chain of elements $v_1<_P\cdots<_P v_m$ is called the \emph{promotion chain of $L$}. When $L$ is a linear extension, the promotion chain is a maximal chain; this is not necessarily the case when $L\not\in\mathcal L(P)$.  (It is also possible to define an extended promotion operator in which we require the $L$-successor of $x$ to cover $x$, but this variant lacks many of the interesting properties of the version that we study here.)

Let $a$ be the largest integer in $\{0,\ldots,n\}$ such that for all $j\in\{n-a+1,\ldots,n\}$, the set $\{x\in P: j\leq L(x)\leq n\}$ forms an upper order ideal in $P$. We say an element $x\in P$ is \emph{frozen with respect to $L$} (or just \emph{frozen} if the labeling $L$ is understood) if $n-a+1\leq L(x)\leq n$. In particular, the set $F$ of frozen elements is an upper order ideal. Notice that the standardization of $L\vert_F$ is a linear extension of the subposet of $P$ induced by $F$. Furthermore, $L$ is a linear extension if and only if $F=P$. 

\begin{example}\label{Exam1}
Consider the following two applications of extended promotion: 
\[\begin{array}{l}\includegraphics[height=4cm]{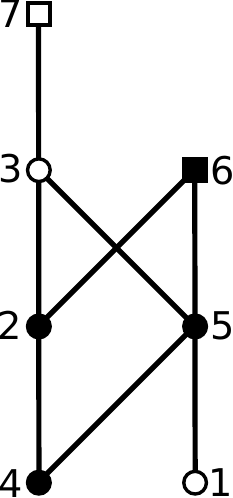}\end{array}\xrightarrow{\,\,\,\partial\,\,\,}\begin{array}{l}\includegraphics[height=4cm]{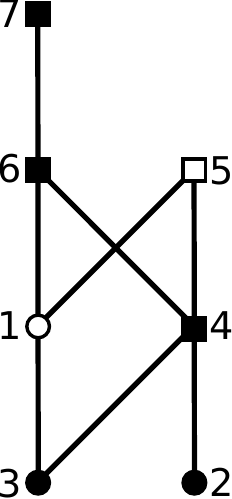}\end{array}\xrightarrow{\,\,\,\partial\,\,\,}\begin{array}{l}\includegraphics[height=4cm]{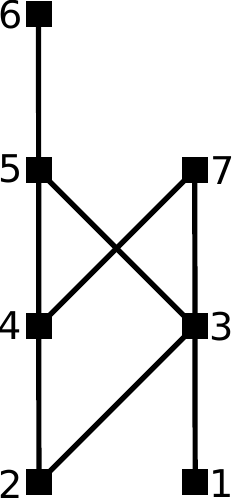}\end{array}\] In each of the first two labelings, we have colored the elements in the promotion chain white. For each labeling, the frozen elements are represented by squares, while the other elements are represented by circles.  
\end{example}

There is a useful alternative description of promotion, discovered independently by Haiman \cite{Haiman} and Malvenuto--Reutenauer \cite{Malvenuto}, that makes use of certain toggle operators (also called Bender-Knuth involutions) that act on linear extensions. We can easily extend these toggle operators $\tau_i$ ($i \in [n-1]$) to arbitrary labelings in order to obtain an alternative description of extended promotion. Suppose $L:P\to[n]$ is a labeling of $P$, and let $i\in[n-1]$. If $L^{-1}(i)<_P L^{-1}(i+1)$, let $\tau_i(L)=L$. Otherwise, let $\tau_i(L):P\to[n]$ be the labeling of $P$ defined by \[\tau_i(L)(x)=\begin{cases} L(x), & \mbox{if } L(x)\not\in\{i,i+1\}; \\ i+1, & \mbox{if } L(x)=i; \\ i, & \mbox{if } L(x)=i+1. \end{cases}\]

\begin{proposition}\label{Prop1}
Let $P$ be an $n$-element poset, and let $L:P\to[n]$ be a labeling of $P$. Then \[\partial(L)=(\tau_{n-1}\circ\cdots\circ \tau_2\circ \tau_1)(L).\] 
\end{proposition}
\begin{proof}
Our proof is an immediate adaptation of the proof given in \cite{StanleyPromotion} for the case in which $L$ is a linear extension. Let us associate each labeling $L':P\to[n]$ with the word $(L')^{-1}(1)\cdots (L')^{-1}(n)$. Let $u_1\cdots u_n$ be the word associated to $L$ (so $L(u_i)=i$ for all $i\in[n]$). Let $\{v_1<_P\cdots<_P v_m\}$ be the promotion chain of the labeling $L$.  Note that $v_1=u_1$. Suppose $m\geq 2$, and let $v_2=u_{r_2}$. By the definition of $v_2$, none of the elements $u_2,\ldots,u_{r_2-1}$ is greater than $v_1$ in $P$. One can now check directly that the word associated to $(\tau_{r_2-1}\circ\cdots\circ \tau_1)(L)$ is $u_2u_3\cdots u_{r_2-1}u_1u_{r_2}u_{r_2+1}\cdots u_n$, which is obtained from $u_1\cdots u_n$ by cyclically shifting the prefix of length $r_2-1$ by $1$ to the left. If $m\geq 3$ and $v_3=u_{r_3}$, then a similar argument shows that the word associated to $(\tau_{r_3-1}\circ\cdots\circ \tau_1)(L)$ is obtained from the word associated to $(\tau_{r_2-1}\circ\cdots\circ\tau_1)(L)$ by cyclically shifting the subword $u_{r_2}\cdots u_{r_3-1}$ by $1$ to the left. Continue in this fashion until reaching $v_m=u_{r_m}$. The word associated to $(\tau_{n-1}\circ\cdots\circ \tau_1)(L)$ is obtained from the word associated to $(\tau_{r_m-1}\circ\cdots\circ \tau_1)(L)$ by cyclically shifting the suffix $u_{r_m}\cdots u_n$ by $1$ to the left. One immediately checks that this agrees with our original definition of extended promotion. 
\end{proof}

\begin{example}
Applying the toggle operators $\tau_1,\ldots,\tau_6$ to the leftmost labeling shown in Example~\ref{Exam1} yields the sequence 
\[\begin{array}{l}\includegraphics[height=4cm]{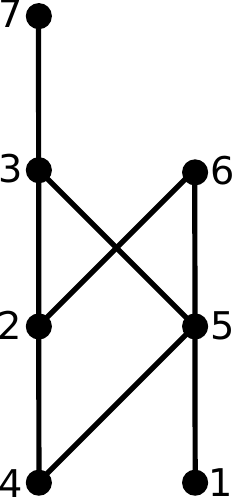}\end{array}\xrightarrow{\,\,\,\tau_1\,\,\,}\begin{array}{l}\includegraphics[height=4cm]{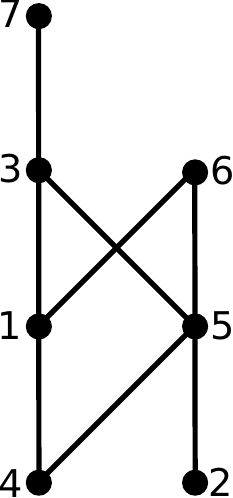}\end{array}\xrightarrow{\,\,\,\tau_2\,\,\,}\begin{array}{l}\includegraphics[height=4cm]{PromotionPIC5}\end{array}\xrightarrow{\,\,\,\tau_3\,\,\,}\begin{array}{l}\includegraphics[height=4cm]{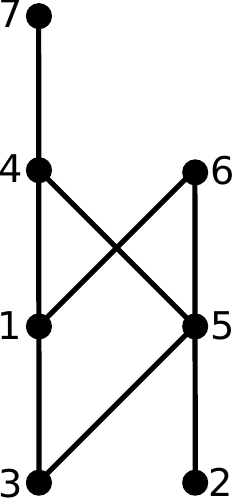}\end{array}\]

\[\xrightarrow{\,\,\,\tau_4\,\,\,}\begin{array}{l}\includegraphics[height=4cm]{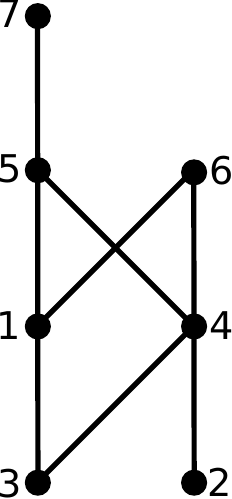}\end{array}\xrightarrow{\,\,\,\tau_5\,\,\,}\begin{array}{l}\includegraphics[height=4cm]{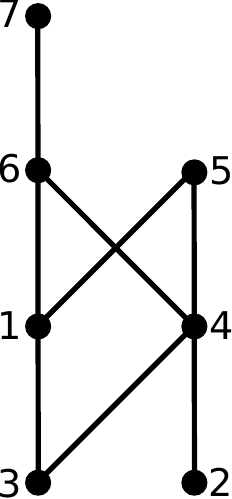}\end{array}\xrightarrow{\,\,\,\tau_6\,\,\,}\begin{array}{l}\includegraphics[height=4cm]{PromotionPIC8}\end{array}.\] The labeling at the end of this sequence agrees with the labeling $\partial(L)$ shown in the middle in Example~\ref{Exam1}, as guaranteed by Proposition~\ref{Prop1}.  
\end{example}

\begin{remark}\label{Rem2}
Using Proposition~\ref{Prop1}, one can show that when $P=\{x_1<_P\cdots<_P x_n\}$ is an $n$-element chain, the extended promotion map $\partial:\Lambda(P)\to\Lambda(P)$ is dynamically equivalent to the bubble sort map {\bf B} studied in \cite{Chung} and \cite[pages 106--110]{Knuth2}. To be more precise, let $S_n$ denote the set of permutations of $[n]$, and let us associate each labeling $L\in\Lambda(P)$ with the permutation $\sigma(L)=L(x_1)\cdots L(x_n)\in S_n$. Let $\widehat\sigma(L)=\sigma(L)^{-1}$ be the inverse of $\sigma(L)$; that is, $\widehat\sigma(L)$ is the permutation in $S_n$ in which $i$ is in position $L(x_i)$ for all $i\in[n]$. Then one can show that $\widehat\sigma(\partial(L))={\bf B}(\widehat\sigma(L))$. 
\end{remark}

We now record a few facts about the interaction between extended promotion and frozen elements.

\begin{lemma}\label{Lem2}
Let $L:P\to[n]$ be a labeling of an $n$-element poset $P$.  Let $F_0$ be the set of elements of $P$ that are frozen with respect to $L$, and let $F_1$ be the set of elements of $P$ that are frozen with respect to $L_1=\partial(L)$. If $L$ is not a linear extension, then $F_0$ is a proper subset of $F_1$. 
\end{lemma}
\begin{proof}
Let $a=|F_0|$, and let $U_0=P\setminus F_0$ be the set of elements of $P$ that are unfrozen with respect to $L$. Assume $L$ is not a linear extension, meaning that $a<n$. By the definition of a frozen element, we must have $F_0=\{x\in P:n-a+1\leq L(x)\leq n\}$. We first prove that the set $X=\{x\in P:n-a\leq L_1(x)\leq n\}$ is an upper order ideal of $P$. Note that $L_{1}^{-1}(1), \ldots, L_{1}^{-1}(n-a-1) \in U_0$.  We claim that the unique element $z \in U_0$ with $L_{1}(z)\geq n-a$ must be a maximal element of the subposet induced by $U_0$.  Indeed, $z$ must be in the promotion chain of $L$, and it must be the case either that $z$ is maximal (in $P$) or that the $L$-successor of $z$ is an element of $F_0$.  If $z$ is maximal in $P$, then certainly $z$ is maximal in $U_0$, as desired.  Otherwise, the definition of the $L$-successor guarantees that every element greater (in $P$) than $z$ is an element of $F_0$, i.e., $z$ is maximal in $U_0$.  We thus conclude that $U_0 \setminus \{z\}$ is a lower order ideal of $P$, which is equivalent to the statement that $X$ is an upper order ideal of $P$. 

We have $X=F_0\cup\{z\}$, so the lemma will follow if we can show that the elements of $X$ are all frozen with respect to $L_1$. Since $X$ is an upper order ideal, we simply need to show that for all $u,v\in X$ such that $u<_P v$, we have $L_1(u)<L_1(v)$. Choose $u,v\in X$ with $u<_P v$. The element $z$ is a minimal element of the subposet induced by $X$, and $v$ is not minimal in $X$ because $u<_P v$. This proves that $v\in X\setminus\{z\}=F_0$, so $v$ is frozen with respect to $L$. By the definition of a frozen element (and the assumption that $u<_P v$), we have $L(u)<L(v)$. If $u$ is not in the promotion chain of $L$, then $L_1(u)=L(u)-1<L(v)-1\leq L_1(v)$, as desired. If $u$ and $v$ are both in the promotion chain, then certainly $L_1(u)<L_1(v)$. Finally, suppose $u$ is in the promotion chain but $v$ is not. Let $u'$ be the $L$-successor of $u$. Since $u<_P v$, the definition of the $L$-successor guarantees that $L(u')<L(v)$. Therefore, $L_1(u)=L(u')-1<L(v)-1=L_1(v)$. 
\end{proof}

\begin{lemma}\label{lem:ideals}
Let $L:P\to[n]$ be a labeling of an $n$-element poset $P$.  For every $0 \leq \gamma \leq n$, the elements $L_\gamma^{-1}(n-\gamma+1),\ldots,L_\gamma^{-1}(n)$ are frozen with respect to $L_\gamma$. 
\end{lemma}
\begin{proof}
It follows immediately from Lemma~\ref{Lem2} that there are at least $\gamma$ elements of $P$ that are frozen with respect to $L_\gamma$. Therefore, the lemma follows from the definition of a frozen element.  
\end{proof}

The following proposition is now an immediate corollary of the preceding lemma. 
\begin{proposition}\label{Prop3}
For every $n$-element poset $P$, we have $\partial^{n-1}(\Lambda(P))=\mathcal L(P)$
\end{proposition}
\begin{proof}
Let $L\in\Lambda(P)$. Set $\gamma=n-1$ in Lemma~\ref{lem:ideals} to see that all of the elements of $P$ with labels at least $2$ are frozen with respect to $L_{n-1}$.  It follows that $L_{n-1}^{-1}(1)$ is minimal in $P$, so we conclude that $L_{n-1}$ is a linear extension. As $L$ was arbitrary, $\partial^{n-1}(\Lambda(P))\subseteq\mathcal L(P)$. On the other hand, the restriction of $\partial$ to $\mathcal L(P)$ is the classical promotion map, which is a bijection from $\mathcal L(P)$ to $\mathcal L(P)$. Thus, $\mathcal L(P)=\partial^{n-1}(\mathcal L(P))\subseteq\partial^{n-1}(\Lambda(P))$. 
\end{proof}

Proposition~\ref{Prop3} motivates the following definition. 
\begin{definition}
Let $P$ be an $n$-element poset. A labeling $L:P\to[n]$ is \emph{$k$-untangled} if the labeling $L_{n-k-2}=\partial^{n-k-2}(L)$ is not a linear extension. We say that $L$ is \emph{tangled} if it is $0$-untangled. We make the convention that $L$ is not $k$-untangled if $n\leq k+1$; in particular, the unique labeling of a $1$-element poset is not tangled.   
\end{definition}

\begin{remark}\label{Rem1}
It follows immediately from Lemma~\ref{Lem2} that a labeling $L$ of an $n$-element poset $P$ is tangled if and only if there are exactly $\gamma$ elements of $P$ that are frozen with respect to $L_\gamma$ for all $0\leq\gamma\leq n-2$.  
\end{remark}

\begin{theorem}\label{Thm1}
Let $0\leq k\leq n-2$. An $n$-element poset $P$ has a $k$-untangled labeling if and only if it has a lower order ideal of size $k+2$ that is not an antichain.
\end{theorem}
\begin{proof}
First, suppose that $P$ has a $k$-untangled labeling $L$, so that $L_{n-k-2} \notin \mathcal{L}(P)$.  For each $j \in [n]$, write $y_j=L_{n-k-2}^{-1}(j)$.  By Lemma~\ref{lem:ideals}, we know that the set $Y=\{y_1, \ldots, y_{k+2}\}$ is a lower order ideal of $P$ and that $L_{n-k-2}\vert_Y$ is not a linear extension.  In particular, $Y$ cannot be an antichain, since every labeling of an antichain is a linear extension.

Second, suppose $P$ has a lower order ideal $Q$ of size $k+2$ that is not an antichain.  We describe a $k$-untangled labeling $L$ of $P$ as follows.  Choose elements $x,y\in Q$ with $x<_P y$.  Let $L(x)=n$ and $L(y)=n-1$.  Assign the labels $n-k-1, n-k, \ldots, n-2$ to the remaining elements of $Q$ arbitrarily, and then assign the labels $1, 2, \ldots, n-k-2$ to the elements of $P \setminus Q$ arbitrarily.  The first $n-k-2$ promotions do not affect the relative order of the labels of the elements of $Q$; more precisely, every $z \in Q$ satisfies $L_{n-k-2}(z)=L(z)-(n-k-2)$.  In particular, we have $L_{n-k-2}(y)=k+1<k+2=L_{n-k-2}(x)$, which shows that $L$ is $k$-untangled.
\end{proof}

The $k=0$ case of Theorem~\ref{Thm1} says that $P$ admits a tangled labeling if and only if there is some element that is greater than exactly $1$ other element.

We end this section with one additional simple lemma.  

\begin{lemma}\label{Lem4}
Let $P$ be an $n$-element poset, and let $L\in\Lambda(P)$. For all $\gamma,j\in[n]$ with $j+\gamma\leq n$, the elements $L^{-1}(j+\gamma)$ and $L_\gamma^{-1}(j)$ belong to the same connected component of $P$.
\end{lemma}
\begin{proof}
It is immediate from the definition of $\partial$ that $L^{-1}(j+\gamma)$ and $L_1^{-1}(j+\gamma-1)$ are in the same connected component of $P$. Similarly, $L_1^{-1}(j+\gamma-1)$ and $L_2^{-1}(j+\gamma-2)$ are in the same connected component. Hence, $L^{-1}(j+\gamma)$ is in the same connected component as $L_2^{-1}(j+\gamma-2)$. Continuing in this fashion, we eventually obtain the desired result.   
\end{proof}

\section{Tangled Labelings of Inflated Rooted Forests}\label{Sec:Enumeration}

In this section, we explicitly enumerate the tangled labelings of a large class of posets that we call inflated rooted forests. 

\begin{definition}
A \emph{rooted forest poset} is a poset in which each element is covered by at most one other element. A \emph{rooted tree poset} is a connected rooted forest poset. 
\end{definition}

Note that a rooted tree poset is just a poset whose Hasse diagram is a rooted tree in which the root is the unique maximal element. A rooted forest poset is a poset whose connected components are rooted tree posets. The following notion of inflation is a generalization of subdivision of edges in a graph. 

\begin{definition}\label{Def1}
Let $Q$ be a finite poset. An \emph{inflation} of $Q$ is a poset $P$ such that there exists a surjective map $\varphi:P\to Q$ with the following properties: 
\begin{enumerate}
\item For each $v\in Q$, the subposet $\varphi^{-1}(v)$ of $P$ has a unique minimal element.  
\item If $x,y\in P$ are such that $\varphi(x)\neq\varphi(y)$, then $x<_P y$ if and only if $\varphi(x)<_Q\varphi(y)$. 
\end{enumerate}
An \emph{inflated rooted forest poset} is a poset that is the inflation of a rooted forest poset. An \emph{inflated rooted tree poset} is a poset that is the inflation of a rooted tree poset. 
\end{definition}

\begin{figure}[ht]
\begin{center}
\includegraphics[height=5.7cm]{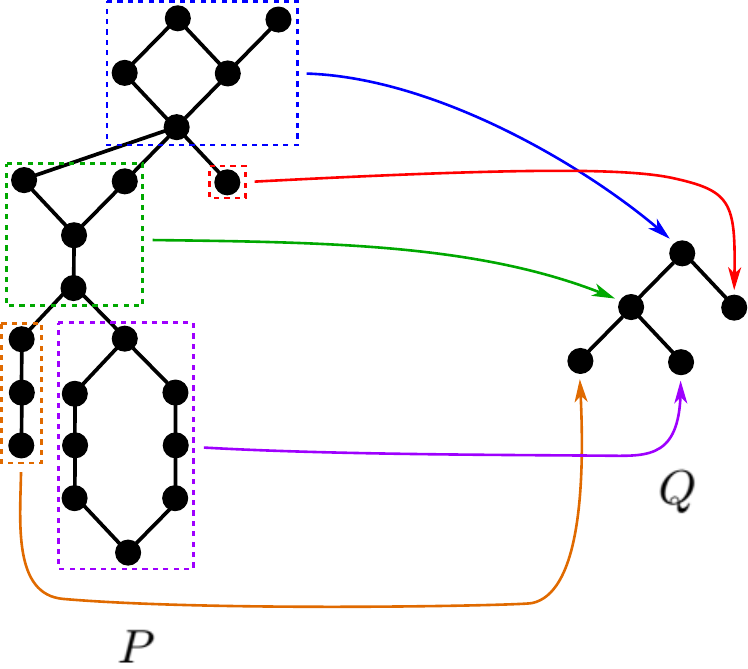}
\caption{The poset $P$ is an inflated rooted tree poset because it is an inflation of the rooted tree poset $Q$. The arrows represent the map $\varphi$ from Definition~\ref{Def1}, and the subsets of $P$ encompassed by the dotted boxes are the preimage sets $\varphi^{-1}(v)$ for $v\in Q$.}
\label{Fig1}
\end{center}  
\end{figure}

\begin{remark}
If $P$ is an inflation of $Q_1$ and the Hasse diagram of $Q_1$ is a subdivision of the Hasse diagram of $Q_2$, then $P$ is an inflation of $Q_2$.  In particular, every inflated rooted tree poset can be obtained as the inflation of a rooted tree poset where no vertex has exactly $1$ child; we will call such a rooted tree poset $Q$ \emph{reduced}.  For the remainder of this section, it will be convenient always to choose $Q$ reduced.
\end{remark}

The following theorem allows us to reduce the problem of enumerating tangled labelings of a poset to the case in which the poset is connected. (The same argument, together with the Principle of Inclusion-Exclusion, can be used to obtain a fairly messy analogous result for $k$-untangled labelings.) In what follows, recall the definition of the standardization map $\st$ from Section~\ref{Subsec:Notation}.

\begin{theorem}\label{thm:connectedreduction}
Let $P$ be an $n$-element poset with connected components $P_1, \ldots, P_r$.  Let $n_i=|P_i|$, and let $t_i$ denote the number of tangled labelings of $P_i$.  A labeling $L$ of $P$ is tangled if and only if there is some $i\in[r]$ such that $\st(L\vert_{P_i})$ is a tangled labeling of $P_i$ and $L^{-1}(n-1),L^{-1}(n) \in P_i$.  Consequently, the number of tangled labelings of $P$ is
\[(n-2)!\sum_{i=1}^r \frac{t_i}{(n_i-2)!}.\]
\end{theorem}
\begin{proof}
First, assume that there exists $i\in[r]$ such that $\st(L\vert_{P_i})$ is tangled and $L^{-1}(n-1),L^{-1}(n)\in P_i$.  For each $0\leq \gamma\leq n-1$, let $P_{\alpha(\gamma)}$ be the connected component of $P$ that contains the element $L_{\gamma}^{-1}(1)$.  By Lemma~\ref{Lem4}, $P_{\alpha(\gamma)}$ is also the connected component of $P$ containing $L^{-1}(\gamma+1)$. In particular, as $\gamma$ ranges from $0$ to $n-3$ (corresponding to the first $n-2$ promotions), the index $\alpha(\gamma)$ assumes the value $i$ exactly $n_i-2$ times. Notice that the promotion chain of $L_\gamma$ is contained in $P_{\alpha(\gamma)}$.  Consequently, when we apply $\partial$ to $L_{\gamma}$, the relative order of the labels of the elements of $P\setminus P_{\alpha(\gamma)}$ does not change. On the other hand, $\st(L_{\gamma+1}\vert_{P_{\alpha(\gamma)}})=\partial(\st(L_{\gamma}\vert_{P_{\alpha(\gamma)}}))$. We find that $\st(L_{n-2}\vert_{P_{i}})=\partial^{n_i-2}(\st(L\vert_{P_{i}}))$, and this labeling is not a linear extension by the assumption that $\st(L\vert_{P_i})$ is tangled.  This means that $L_{n-2}$ is also not a linear extension, so $L$ is tangled.

For the converse, suppose $L$ is tangled. This means that the labeling $L_{n-2}$ is not a linear extension, so there exist $x,y\in P$ such that $x<_Py$ and $L_{n-2}(x)>L_{n-2}(y)$. Lemma~\ref{lem:ideals} tells us that the elements $L_{n-2}^{-1}(j)$ with $3\leq j\leq n$ are frozen, so we must have $L_{n-2}(x)=2$ and $L_{n-2}(y)=1$. In particular, $L_{n-2}^{-1}(1)$ and $L_{n-2}^{-1}(2)$ belong to the same connected component of $P$, say, $P_i$. Furthermore, $\st(L_{n-2}\vert_{P_i})$ is not a linear extension of $P_i$. Lemma~\ref{Lem4} tells us that $L^{-1}(n-1)$ is in the same connected component as $L_{n-2}^{-1}(1)$ and that $L^{-1}(n)$ is in the same connected component as $L_{n-2}^{-1}(2)$. This implies that $L^{-1}(n-1)$ and $L^{-1}(n)$ both belong to $P_i$. With $\alpha(\gamma)$ as in the previous paragraph, we see that $i$ appears exactly $n_i-2$ times in the list $\alpha(0),\alpha(1),\ldots,\alpha(n-3)$. This means that $\st(L_{n-2}\vert_{P_i})=\partial^{n_i-2}(L\vert_{P_i})$, so $\partial^{n_i-2}(L\vert_{P_i})$ is not a linear extension of $P_i$. Hence, $L\vert_{P_i}$ is tangled.    

We now prove the second part of the proposition, which enumerates the tangled labelings of $P$ in terms of the enumeration of the tangled labelings of the connected components of $P$. To construct a tangled labeling $L$ of $P$, we must first choose a connected component $P_i$ that will contain $L^{-1}(n-1)$ and $L^{-1}(n)$. We must then choose the $n_i-2$ elements of $L(P_i)\setminus\{n-1,n\}$; the number of ways to make this choice is $\binom{n-2}{n_i-2}$. We then have to assign these labels to the elements of $P_i$ so that $\st(L\vert_{P_i})$ is a tangled labeling of $P_i$; the number of ways to do this is $t_i$. We then simply assign the remaining $n-n_i$ labels to the elements of $P\setminus P_i$ arbitrarily; the number of ways to do this is $(n-n_i)!$.  Consequently, the total number of tangled labelings of $P$ is
\[\sum_{i=1}^r t_i\binom{n-2}{n_i-2}(n-n_i)!=(n-2)!\sum_{i=1}^r \frac{t_i}{(n_i-2)!}. \qedhere\]
\end{proof}

Theorem~\ref{thm:connectedreduction} tells us that in order to enumerate tangled labelings of inflated rooted forest posets, it suffices to consider inflated rooted tree posets. Before doing so, we require some more terminology.  Let $Q$ be a rooted tree poset with minimal elements (leaves) $\ell_1, \ldots, \ell_s$.  For each $\ell_i$, there is a unique path in the Hasse diagram of $Q$ from $\ell_i$ to the maximal element (root).  Call the elements of this path $u_{i,0}, u_{i,1}, \ldots, u_{i,\omega(i)}$, where $u_{i,0}=\ell_i$ and for all $1\leq j\leq\omega(i)$, $u_{i,j}$ covers $u_{i,j-1}$ (so that $u_{i,\omega(i)}$ is the root of $Q$).  Furthermore, suppose $P$ is an inflation of $Q$ with map $\varphi$.  Then we define \[b_{i,j}=\sum_{v\leq_Q u_{i,j-1}}|\varphi^{-1}(v)|\quad\text{and}\quad c_{i,j}=\sum_{v<_Q u_{i,j}}|\varphi^{-1}(v)|.\]
The quotient $\frac{b_{i,j}}{c_{i,j}}$ represents, among the elements of $P$ that are smaller than the minimal element of $\varphi^{-1}(u_{i,j})$, the fraction of those that lie ``in the direction'' of $\varphi^{-1}(\ell_i)$.

Recall that, by convention, a $1$-element poset has no tangled labelings. In what follows, it will be convenient to restrict our attention to inflated rooted tree posets $P$ with $|P|\geq 2$. 

\begin{theorem}\label{thm:inflated-tree}
Let $P$ be an inflation of the (reduced) rooted tree poset $Q$ with map $\varphi:P\to Q$, and assume $|P|\geq 2$. Let $\ell_1, \ldots, \ell_s$ be the leaves of $Q$, and let the $u_{i,j}$'s, $b_{i,j}$'s, and $c_{i,j}$'s be as defined above.  Then the number of tangled labelings of $P$ is
$$(n-1)!\sum_{i=1}^s \prod_{j=1}^{\omega(i)} \frac{b_{i,j}-1}{c_{i,j}-1},$$
where $n=\displaystyle\sum_{v \in Q}|\varphi^{-1}(v)|$ is the total size of $P$.
\end{theorem}

(Strictly speaking, this formula also holds for $Q$ not reduced if we use the convention $0/0=1$; the proof in this more general setting requires a few extra case distinctions.  However, since every inflated rooted tree poset is the inflation of a reduced rooted tree poset, we deem it cleaner to avoid this issue entirely.)

A \emph{rooted star poset} is a rooted tree poset in which the root covers every leaf. The Hasse diagram of a rooted tree poset is graph-theoretically isomorphic to a star graph, where the center of the star is the root of the tree. The preceding theorem gives a particularly clean formula for the number of tangled labelings of an inflation of a rooted star poset.

\begin{corollary}\label{corollary:corolla}
Let $Q$ be a rooted star poset with root $v^*$, and let $s\geq 1$ be the number of children of $v^*$. Let $P$ be an inflation of $Q$ with map $\varphi$. The number of tangled labelings of $P$ is
$$(n-1)! \left(\frac{\mu-s}{\mu-1}\right),$$
where $n=|P|$ and $\mu=n-|\varphi^{-1}(v^*)|$.
\end{corollary}
\begin{proof}
Write $\ell_1, \ldots, \ell_s$ for the leaves of $Q$.  For each $1\leq i\leq s$ we have $\omega(i)=1$, $u_{i,0}=\ell_i$, $u_{i,1}=v^*$, $b_{i,1}=|\varphi^{-1}(\ell_i)|$, and $c_{i,1}=\sum_{k=1}^s |\varphi^{-1}(\ell_k)|=\mu$.  Then the formula in Theorem~\ref{thm:inflated-tree} reduces to
$$(n-1)! \sum_{i=1}^s \frac{|\varphi^{-1}(\ell_i)|-1}{\mu-1}=(n-1)! \left(\frac{\mu-s}{\mu-1}\right),$$
as needed.
\end{proof}

We also obtain a simple formula for the number of tangled labelings of a poset in which every connected component has a unique minimal element. 
\begin{corollary}\label{Cor1}
Let $P$ be an $n$-element poset with $r$ connected components, each of which has a unique minimal element. The number of tangled labelings of $P$ is \[(n-r)(n-2)!.\]
\end{corollary}

\begin{proof}
We first consider the case in which $r=1$ and $n\geq 2$. This means that $P$ has a unique minimal element, so it is an inflation of a $1$-element poset $Q$. In the notation of Theorem~\ref{thm:inflated-tree}, we have $s=1$ and $\omega(1)=0$. Then the product appearing in the formula in that theorem is the empty product $1$, so the number of tangled labelings of $P$ is $(n-1)!=(n-1)(n-2)!$. If $r=1$ and $n=1$, then $P$ has no tangled labelings by convention. In this case, the number of tangled labelings is still $(n-1)(n-2)!$ if we make the (unusual) convention $(-1)!=1$. 

Now suppose $r\geq 2$. Let $P_1,\ldots P_r$ be the connected components of $P$. Let $n_i=|P_i|$, and let $t_i$ be the number of tangled labelings of $P_i$. The preceding paragraph tells us that $t_i=(n_i-1)(n_i-2)!$. Theorem~\ref{thm:connectedreduction} now implies that the number of tangled labelings of $P$ is  \[(n-2)!\sum_{i=1}^r\frac{t_i}{(n_i-2)!}=(n-2)!\sum_{i=1}^r(n_i-1)=(n-r)(n-2)!. \qedhere\]
\end{proof}

The proof of Theorem~\ref{thm:inflated-tree} begins with the following lemma. 

\begin{lemma}\label{lem:n-minimal}
If $L$ is a tangled labeling of an $n$-element poset $P$, then $L^{-1}(n)$ is a minimal element of $P$.
\end{lemma}
\begin{proof}
Let $F_\gamma$ be the set of elements of $P$ that are frozen with respect to $L_\gamma$. Because $L$ is tangled, Remark~\ref{Rem1} tells us that $|F_\gamma|=\gamma$ for all $0\leq\gamma\leq n-2$. Consequently, $F_\gamma=\{x\in P:n-\gamma+1\leq L_\gamma(x)\leq n\}$. This means that $L_\gamma^{-1}(n-\gamma)$ is not frozen with respect to $L_\gamma$, so the set $F_\gamma\cup\{L_\gamma^{-1}(n-\gamma)\}$ is not an upper order ideal of $P$. The set $F_\gamma$ is an upper order ideal of $P$, so there must be some $y\in P$ such that $L_\gamma^{-1}(n-\gamma)<_P y$ and $L(y)<n-\gamma$. This implies that $L_\gamma^{-1}(n-\gamma)$ is not in the promotion chain of $L_\gamma$, so $L_\gamma^{-1}(n-\gamma)=L_{\gamma+1}^{-1}(n-\gamma-1)$. As this is true for all $0\leq\gamma\leq n-2$, we must have $L^{-1}(n)=L_1^{-1}(n-1)=L_2^{-1}(n-2)=\cdots=L_{n-1}^{-1}(1)$. However, we know that $L_{n-1}$ is a linear extension of $P$, so $L_{n-1}^{-1}(1)$ is a minimal element of $P$. 
\end{proof}

Let us sketch the argument used in the proof of Theorem~\ref{thm:inflated-tree} before presenting it formally.  We imagine $L(=L_0),L_1,L_2,\ldots,L_{n-2}$ as a sequence of labelings; in each $L_{\gamma}$, we pay particular attention to the ``special'' element with the label $n-1-\gamma$.  Lemma~\ref{lem:n-minimal} tells us that in a tangled labeling, the element with the label $n$ must be minimal; for each such minimal element, we ask how many labelings have the property that the special element always stays above this minimal element as we iteratively apply $\partial$.  In the case of an inflated rooted tree poset, all that matters is which subtree the special element ``slides down into'' at each ``branch vertex,'' and it turns out that (under certain technical conditions) this distribution is proportional to the distribution of the sizes of the subtrees.  We begin the formal discussion by developing a few preparatory tools.

\begin{lemma}\label{Lem3}
Let $P$ be an $N$-element poset, and let $X=\{y\in P:y<_P x\}$ for some $x\in P$. Suppose that every element of $P$ that is comparable with some element of $X$ is also comparable with $x$. If $L$ and $\widetilde L$ are labelings of $P$ that agree on $P\setminus X$, then for every $\gamma\geq 1$, the labelings $L_\gamma$ and $\widetilde L_\gamma$ also agree on $P\setminus X$.
\end{lemma}
\begin{proof}
It suffices to prove the case in which $\gamma=1$; the general case will then follow by induction.  Let $C$ and $\widetilde C$ be the promotion chains of $L$ and $\widetilde L$, respectively. It suffices to prove that $C\cap(P\setminus X)=\widetilde C\cap(P\setminus X)$. Since $L$ and $\widetilde L$ agree on $P\setminus X$, it suffices to show that the minimal elements of $C\cap(P\setminus X)$ and $\widetilde C\cap(P\setminus X)$ are the same. Notice that $L(X)=\widetilde L(X)$. If $1\not\in L(X)$, then the minimal element of $C\cap(P\setminus X)$ is $L^{-1}(1)$. This is also $\widetilde L^{-1}(1)$, which is the minimal element of $\widetilde C\cap(P\setminus X)$. Now suppose $1\in L(X)$. The minimal element of $C\cap(P\setminus X)$ must be the element of $\{y\in P:y\geq_P x\}$ with the smallest label. This is also the minimal element of $\widetilde C\cap(P\setminus X)$. 
\end{proof}

The following lemma is an immediate consequence of Lemma~\ref{Lem3}. 

\begin{lemma}\label{prop:partition}
Let $P$, $x$, and $X$ be as in Lemma~\ref{Lem3}. If $L$ is a labeling of $P$ and $\gamma\geq 0$, then the set $L_\gamma(X)$ depends on only the set $L(X)$ and the restriction $L\vert_{P\setminus X}$; it does not depend on the way in which the labels in $L(X)$ are distributed among the elements of $X$. 
\end{lemma}

\begin{lemma}\label{lem:fork-disribution}
Let $P$, $x$, and $X$ be as in Lemma~\ref{Lem3}, and assume $X$ is nonempty.  Further suppose that there is a partition of $X$ into disjoint subsets $A$ and $B$ such that no element of $A$ is comparable with any element of $B$.  Fix an injective map $M:P \setminus X\to[N]$ such that every labeling $L:P\to[N]$ satisfying $L\vert_{P\setminus X}=M$ has the property that $L_{N-1}^{-1}(1) \in X$.  If a labeling $L:P\to[N]$ satisfying $L\vert_{P\setminus X}=M$ is chosen uniformly at random, then the probability that $L_{N-1}^{-1}(1) \in A$ is $|A|/|X|$.
\end{lemma}

\begin{proof}
For each $0\leq \gamma\leq N-2$, we have $L_\gamma^{-1}(N-\gamma)\geq_P L_{\gamma+1}^{-1}(N-\gamma-1)$. Notice that $A$ and $B$ are lower order ideals of $P$. Suppose first that $L^{-1}(N) \in X$. In this case, $L^{-1}_{N-1}(1) \in A$ if and only if $L^{-1}(N) \in A$, where the latter event happens with probability $|A|/|X|$.  For the remainder of the proof, we consider the case where $L^{-1}(N) \not\in X$.

By hypothesis, $L_{N-1}^{-1}(1)\in X$. Therefore, there exists a largest index $\gamma_1$ such that $L_{\gamma_1}^{-1}(N-\gamma_1) \notin X$.  By Lemma~\ref{prop:partition}, the value of $\gamma_1$ does not depend on how $L$ extends $M$.  Of course, we have $L_{\gamma_1+1}^{-1}(N-\gamma_1-1) \in X$ by the maximality of $\gamma_1$.  This implies that $L_{\gamma_1}^{-1}(1)\in X$. We now see that $L^{-1}_{\gamma_1+1}(N-\gamma_1-1) \in A$ if $L_{\gamma_1}^{-1}(1) \in A$ and that $L^{-1}_{\gamma_1+1}(N-\gamma_1-1) \in B$ if $L_{\gamma_1}^{-1}(1) \in B$.  Our problem is thus reduced to determining whether $L_{\gamma_1}^{-1}(1)$ is in $A$ or $B$.

If $L^{-1}(\gamma_1+1) \in X$ (whether or not this occurs depends only on $M$), then $L_{\gamma_1}^{-1}(1)\in A$ if and only if $L^{-1}(\gamma_1+1)\in A$. In this case, the probability that $L^{-1}(\gamma_1+1) \in A$ (which is equivalent to $L^{-1}_{N-1}(1) \in A$) is precisely $|A|/|X|$. We now restrict our attention to the case where $L^{-1}(\gamma_1+1) \not\in X$.  Since $L_{\gamma_1}^{-1}(1)\in X$, there must exist a largest index $\gamma_2<\gamma_1$ such that $L_{\gamma_2}^{-1}(\gamma_1+1-\gamma_2) \notin X$.  Analogously to above, we have $L_{\gamma_2+1}^{-1}(\gamma_1-\gamma_2) \in X$. Moreover, $L_{\gamma_2+1}^{-1}(\gamma_1-\gamma_2) \in A$ if and only if $L_{\gamma_2}^{-1}(1) \in A$.  So our problem is reduced to determining whether $L_{\gamma_2}^{-1}(1)$ is in $A$ or $B$.

We have $L_{\gamma_2}^{-1}(1)\in X$. As before, if $L^{-1}(\gamma_2+1) \in X$ (where the occurrence of this event depends only on $M$), then $L_{\gamma_2}^{-1}(1)\in A$ if and only if $L^{-1}(\gamma_2+1)\in A$. In this case, the probability that $L^{-1}(\gamma_2+1) \in A$ (which is equivalent to $L^{-1}_{N-1}(1) \in A$) is precisely $|A|/|X|$.  Otherwise, we define $\gamma_3<\gamma_2$ to be the largest index such that $L_{\gamma_3}^{-1}(\gamma_2+1-\gamma_3)\not\in X$. We continue in this manner: at each step, if $L^{-1}(\gamma_i+1)\not\in X$, then we define $\gamma_{i+1}<\gamma_i$ to be the largest index such that $L_{\gamma_{i+1}}^{-1}(\gamma_i+1-\gamma_{i+1})\not\in X$.  This process must end after finitely many steps because the $\gamma_i$'s are strictly decreasing nonnegative integers.  Regardless of how many steps this process takes to terminate, the probability that $L^{-1}_{N-1}(1) \in A$ is $|A|/|X|$, as desired.
\end{proof}

Putting these pieces together lets us complete the proof of Theorem~\ref{thm:inflated-tree}.

\begin{figure}[ht]
\begin{center}
\includegraphics[height=5.7cm]{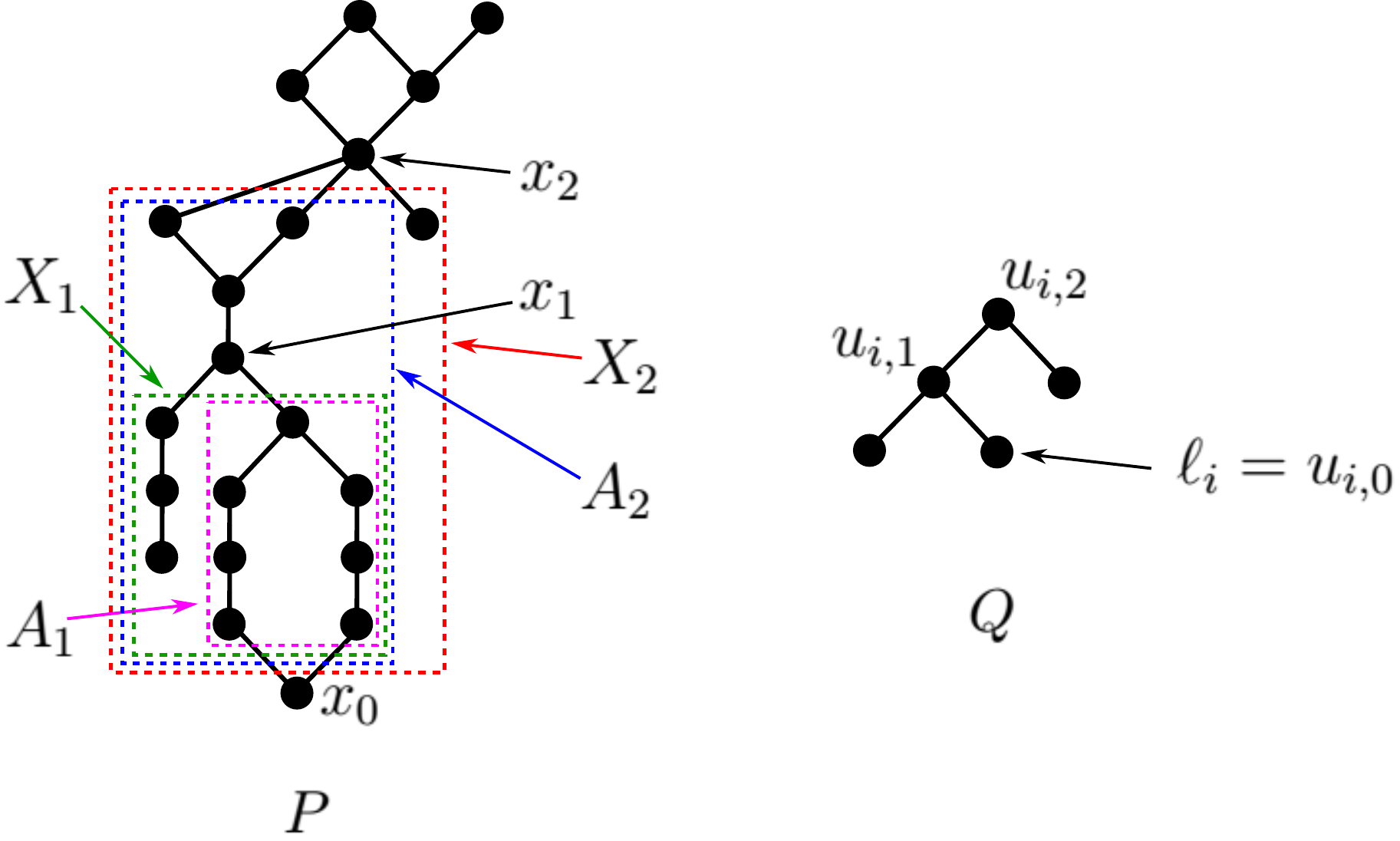}
\caption{An illustration of the notation used in the proof of Theorem~\ref{thm:inflated-tree} with the posets $P$ and $Q$ from Figure~\ref{Fig1}. We have fixed an arbitrary leaf $\ell_i$ and labeled the rest of the figure accordingly. In this example, $\omega(i)=2$. }
\label{Fig2}
\end{center}  
\end{figure}

\begin{proof}[Proof of Theorem~\ref{thm:inflated-tree}]
Fix a leaf $\ell_i$; we will count the tangled labelings $L$ such that the (unique) minimal element $x_0$ of $\varphi^{-1}(\ell_i)$ has the label $n$.  Note that by the proof of Lemma~\ref{lem:n-minimal}, we have $L_{n-2}(x_0)=2$. We know by Lemma~\ref{lem:ideals} that $L$ is tangled if and only if $L_{n-2}^{-1}(1)>_Px_0$. Because $Q$ is reduced, this occurs if and only if $L_{n-2}^{-1}(1) \in \varphi^{-1}(\ell_i)$.

Note that $\omega(i) \geq 1$ unless $Q$ is a $1$-element poset.  In this case (in which there is only $1$ leaf $\ell_i$ anyway), $\omega(\ell_i)=0$ and $\varphi^{-1}(\ell_i)=P$; it then follows that all labelings $L$ with $L(x_0)=n$ are tangled, and there are $(n-1)!$ such labelings.  For the remainder of this proof, we assume that $Q$ has at least $2$ elements, so that $\omega(i) \geq 1$.

Let $\widetilde P=P \setminus \{x_0\}$, and let $\widetilde{L}=L\vert_{\widetilde P}$.  Note that for $0 \leq \gamma \leq n-2$, we have $\st(L_{\gamma}\vert_{\widetilde P})=\widetilde{L}_{\gamma}$.  In particular, $L_{n-2}^{-1}(1) \in \varphi^{-1}(\ell_i)$ if and only if $\widetilde{L}_{n-2}^{-1}(1) \in \widetilde{\varphi}^{-1}(\ell_i)$, where $\widetilde{\varphi}=\varphi \vert_{\widetilde{P}}$.  We now choose $L$ uniformly at random among the $(n-1)!$ labelings $L: P \to [n]$ with $L(x_0)=n$, which induces the uniform distribution on labelings $\widetilde{L}: \widetilde{P} \to [n-1]$.

For each $1 \leq j \leq \omega(i)$, let $x_j$ be the unique minimal element of $\widetilde{\varphi}^{-1}(u_{i,j})$. Let \[X_j=\{y \in \widetilde{P}: y<_{\widetilde{P}} x_j\}\quad\text{and}\quad A_j=\bigcup_{v \leq_Q u_{i,j-1}} \widetilde{\varphi}^{-1}(v)\] (see Figure~\ref{Fig2}). Fix any injective map $$M_{\omega(i)}: \widetilde{P} \setminus X_{\omega(i)} \to [n-1],$$ and note that every labeling $\widetilde{L}:\widetilde{P} \to [n-1]$ with $\widetilde{L} \vert_{\widetilde{P} \setminus X_{\omega(i)}}=M_{\omega(i)}$ satisfies $\widetilde{L}_{n-2}^{-1}(1) \in X_{\omega(i)}$. Indeed, every minimal element of $P$ is in $X_{\omega(i)}$ since $x_{\omega(i)}$ is the minimal element of $\varphi^{-1}(u_{i,\omega(i)})$, where $u_{i,\omega(i)}$ is the root of $Q$ (and $|Q|\geq 2$). 
Applying Lemma~\ref{lem:fork-disribution} with $N=n-1$, $x=x_{\omega(i)}$, $X=X_{\omega(i)}$, $A=A_{\omega(i)}$, and $M=M_{\omega(i)}$, we see that the probability that $\widetilde{L}_{n-2}^{-1}(1) \in A_{\omega(i)}$ is
$$\frac{|A_{\omega(i)}|}{|X_{\omega(i)}|}=\frac{b_{i, \omega(i)}-1}{c_{i, \omega(i)}-1},$$
and Lemma~\ref{prop:partition} tells us that the occurrence of this event depends only on $L \vert_{\widetilde{P} \setminus A_{\omega(i)}}$.

Now, fix any injective map $$M_{\omega(i)-1}: \widetilde{P} \setminus X_{\omega(i)-1} \to [n-1]$$ such that every labeling $\widetilde{L}:\widetilde{P} \to [n-1]$ with $\widetilde{L} \vert_{\widetilde{P} \setminus X_{\omega(i)-1}}=M_{\omega(i)-1}$ satisfies $\widetilde{L}_{n-2}^{-1}(1) \in A_{\omega(i)}$ (and hence also $\widetilde{L}_{n-2}^{-1}(1) \in X_{\omega(i)-1}$).  We now consider the uniform distribution on such labelings $\widetilde{L}:\widetilde{P} \to [n-1]$ with $\widetilde{L} \vert_{\widetilde{P} \setminus X_{\omega(i)-1}}=M_{\omega(i)-1}$.  Another application of Lemma~\ref{lem:fork-disribution}, this time with $x=x_{\omega(i)-1}$, $X=X_{\omega(i)-1}$, $A=A_{\omega(i)-1}$, and $M=M_{\omega(i)-1}$, tells us that among these labelings, the probability that $\widetilde{L}_{n-2}^{-1}(1) \in A_{\omega(i)-1}$ is
$$\frac{|A_{\omega(i)-1}|}{|X_{\omega(i)-1}|}=\frac{b_{i, \omega(i)-1}-1}{c_{i, \omega(i)-1}-1},$$
where once again Lemma~\ref{prop:partition} tells us that the occurrence of this event depends only on $L \vert_{\widetilde{P} \setminus A_{\omega(i)-1}}$.  Continuing this process down to $j=1$ yields that the probability of $\widetilde{L}_{n-2}^{-1}(1) \in A_1= \widetilde\varphi^{-1}(\ell_i)$ is
$$\prod_{j=1}^{\omega(i)} \frac{b_{i,j}-1}{c_{i,j}-1},$$
and summing over the leaves of $P$ gives the result.
\end{proof}

\section{Sortable Labelings}\label{sec:sortable}
One of the most fundamental problems concerning any sorting procedure is that of enumerating the objects that require only one iteration of the procedure to become sorted. Hence, we define a labeling $L$ of a finite poset $P$ to be \emph{sortable} if $\partial(L)$ is a linear extension. The set of sortable labelings is precisely the set of preimages under $\partial$ of the set of linear extensions; note that this contains the linear extensions themselves. Consequently, it will be useful to have a way of understanding the preimages of an arbitrary labeling under $\partial$. 

Given a labeling $L$ of a poset $P$, we say an element $x$ of $P$ is \emph{$L$-golden} if for every $y\in P$ with $y>_Px$, we have $L(y)>L(x)$. We say a chain in $P$ is $L$-golden if all of its elements are $L$-golden. 

\begin{proposition}\label{Prop4}
Let $P$ be an $n$-element poset. A labeling $L$ of $P$ is in the image of $\partial$ if and only if $L^{-1}(n)$ is a maximal element of $P$. If $L$ is in the image of $\partial$, then $|\partial^{-1}(L)|$ is equal to the number of $L$-golden chains of $P$ that contain $L^{-1}(n)$. 
\end{proposition}

\begin{proof}
It is clear that every labeling in the image of $\partial$ assigns the label $n$ to a maximal element of $\partial$. Now suppose $L$ is a labeling such that $L^{-1}(n)$ is maximal in $P$. We will prove that $|\partial^{-1}(L)|$ is the number of $L$-golden chains of $P$ that contain $L^{-1}(n)$. One such golden chain is the singleton chain $\{L^{-1}(n)\}$, so this will prove that $L$ is in the image of $\partial$. 

Suppose $L'\in\partial^{-1}(L)$, and let $C_{L'}=\{v_1<_P \cdots <_P v_m\}$ be the promotion chain of $L'$. Consider an index $i\in\{1,\ldots, m-1\}$, and let $y\in P$ be such that $y>_P v_i$. If $y\notin C_{L'}$, then $L'(v_{i+1})<L'(y)$ because $v_{i+1}$ is the $L'$-successor of $v_i$. In this case, $L(v_i)=L'(v_{i+1})-1<L'(y)-1=L(y)$. On the other hand, if $y\in C_{L'}$, then $L(y)>L(v_i)$ because the labels of the promotion chain remain in increasing order after we apply $\partial$ to $L'$. This proves that $v_i$ is $L$-golden. Notice also that $v_m=L^{-1}(n)$ is $L$-golden because it is maximal in $P$. Hence, $C_{L'}$ is an $L$-golden chain containing $L^{-1}(n)$. Thus, we have a map $L'\mapsto C_{L'}$ from $\partial^{-1}(L)$ to the set of $L$-golden chains containing $L^{-1}(n)$. This map is certainly injective because a labeling $L'$ is determined by its promotion chain and the labeling $\partial(L')$. 

It remains to show that the above map from $\partial^{-1}(L)$ to the set of $L$-golden chains containing $L^{-1}(n)$ is surjective, To this end, choose an $L$-golden chain $C=\{v_1<_P\cdots<_Pv_m\}$ with $L^{-1}(n)\in C$. We must have $L(v_m)=n$. For each $x\in P$, let \[L'(x)=\begin{cases} L(x)+1, & \mbox{if } x\not\in C; \\ L(v_{i-1})+1, & \mbox{if } x=v_i\text{ for some }i\in\{2,\ldots,m\}; \\ 1, & \mbox{if } x=v_1. \end{cases}\] It is straightforward to check that $L'$ is a labeling in $\partial^{-1}(L)$ such that $C_{L'}=C$. 
\end{proof}

\begin{theorem}
Let $P$ be an $n$-element poset, and let $\mathcal M$ denote the set of maximal elements of $P$. For each $x\in \mathcal M$, let $\mathscr C_x$ be the number of chains of $P$ that contain $x$. The number of sortable labelings of $P$ is \[\sum_{x\in\mathcal M}\mathscr C_{x}|\mathcal L(P\setminus\{x\})|.\]
\end{theorem}

\begin{proof}
For each $x\in\mathcal M$, let $\mathcal L_x(P)$ be the set of linear extensions $L$ of $P$ such that $L(x)=n$. The collection $\{\mathcal L_x(P):x\in\mathcal M\}$ is a partition of $\mathcal L(P)$, so the number of sortable labelings of $P$ is \[|\partial^{-1}(\mathcal L(P))|=\sum_{x\in\mathcal M}|\partial^{-1}(\mathcal L_x(P))|.\] Notice that if $L$ is a linear extension of $P$, then every element of $P$ is $L$-golden. This implies that every chain of $P$ is $L$-golden, so $\mathscr C_x$ is the number of $L$-golden chains containing $x$. Consequently, it follows from Proposition~\ref{Prop4} that $|\partial^{-1}(L)|=\mathscr C_x$ for every $L\in\mathcal L_x(P)$. Hence, the number of sortable labelings of $P$ is $\sum_{x\in\mathcal M}\mathscr C_x|\mathcal L_x(P)|$. The desired result now follows from the observation that the map $L\mapsto L\vert_{P\setminus\{x\}}$ is a bijection from $\mathcal L_x(P)$ to $\mathcal L(P\setminus\{x\})$.
\end{proof}

\section{Open Problems}\label{sec:conclusion}
We have only scratched the surface of the investigation of extended promotion. In light of Theorem~\ref{thm:inflated-tree}, it would be nice to enumerate tangled labelings of other posets that are not inflated rooted forests.  We have tried to prove theorems that apply to large classes of posets, but it could also be interesting to consider more intricate questions for narrower classes of posets. For example, it is natural to ask for the expected number of iterations of $\partial$ needed to send a random labeling of a poset $P$ to a linear extension. We expect that obtaining specific information about this expected value for a vast collection of posets could be quite difficult. However, it would still be interesting to find nontrivial estimates (or even exact values) of these expected values for some narrow classes of posets. Two classes of posets that are natural candidates are products of chains and rooted tree posets.

Usually, when faced with a sorting map, one tries to understand the objects that are sorted via one iteration, as we did in Section~\ref{sec:sortable}. The next natural problem is then to understand the objects that can be sorted with two iterations. Let us say a labeling $L$ of a poset $P$ is \emph{2-promotion-sortable} if $\partial^2(L)\in\mathcal L(P)$. It would be interesting to explicitly enumerate $2$-promotion-sortable labelings for some specific classes of posets such as products of chains or rooted tree posets.  This is, of course, equivalent to asking for the number of preimages of the sortable labelings. It could also be interesting to find general estimates for the number of $2$-promotion-sortable labelings of arbitrary posets.  

The following conjecture is motivated by Corollary~\ref{Cor1}, which tells us that if $P$ is an $n$-element poset with a unique minimal element, then the number of tangled labelings of $P$ is $(n-1)!$. 

\begin{conjecture} 
If $P$ is an $n$-element poset, then the number of tangled labelings of $P$ is at most $(n-1)!$. 
\end{conjecture}

Recall that a sequence of real numbers $a_0,a_1,\ldots,a_{n-1}$ is called \emph{unimodal} if there is an index $j$ such that $a_0\leq a_1\leq\cdots\leq a_{j-1}\leq a_j\geq a_{j+1}\geq\cdots a_{n-1}$. We say this sequence is \emph{log-concave} if $a_{i-1}a_{i+1}\leq a_i^2$ for all $i\in\{1,\ldots,n-2\}$. Let $P$ be an $n$-element poset, and let $a_k(P)$ denote the number of labelings $L\in\Lambda(P)$ such that $\partial^k(L)\in\mathcal L(P)$. Let $\widehat a_0(P)=a_0(P)=|\mathcal L(P)|$. For $k\geq 1$, let $\widehat a_k(P)=a_k(P)-a_{k-1}(P)$. 

\begin{conjecture}\label{Conj1}
Let $P$ be an $n$-element poset. The sequence $\widehat a_0(P),\widehat a_1(P),\ldots,\widehat a_{n-1}(P)$ is unimodal.  
\end{conjecture}

\begin{question}
Is it true that for every $n$-element poset $P$, the sequence $a_0(P),a_1(P),\ldots,a_{n-1}(P)$ is log-concave? 
\end{question}

When $P$ is an $n$-element chain, it follows from Remark~\ref{Rem2} that $a_k(P)$ is the number of permutations $\pi\in S_n$ such that ${\bf B}^k(\pi)=123\cdots n$, where ${\bf B}$ is the bubble sort map. It follows from the results in \cite[pages 106--110]{Knuth2} that $a_k(P)=(k+1)^{n-k-1}(k+1)!$ for each $1\leq k\leq n-1$. Using this, one can prove that Conjecture~\ref{Conj1} holds when $P$ is a chain. Furthermore, the sequence $a_0(P),a_1(P),\ldots,a_{n-1}(P)$ is log-concave.  

\section*{Acknowledgements}\label{sec:acknowledgements}

We are grateful to James Propp for suggesting the idea to define and investigating an extension of promotion. We also thank Brice Huang for engaging in helpful discussions about previous work on promotion.  The first author was supported by a Fannie and John Hertz Foundation Fellowship and an NSF Graduate Research Fellowship.

\end{document}